\documentclass[12pt]{amsart}
\usepackage{amssymb,amsmath,amsthm,amssymb, amsfonts, amscd}
\usepackage{graphicx}
\usepackage{color}
\usepackage[all]{xy}
\usepackage{anysize}
\usepackage[cp1252]{inputenc}
\usepackage[english]{babel}
\usepackage{pdfpages}
\usepackage{hyperref}

\newtheorem{theorem}{Theorem}[section]
\newtheorem{lemma}[theorem]{Lemma}

\newtheorem{corollary}[theorem]{Corollary}
\newtheorem{proposition}[theorem]{Proposition}
\newtheorem{remark}[theorem]{Remark}
\newtheorem{example}[theorem]{Example}

\DeclareMathAlphabet{\mathpzc}{OT1}{pzc}{m}{it}

\newcommand{\R}{\mathbb{R}}
\newcommand{\N}{\mathbb{N}}
\newcommand{\C}{\mathbb{C}}
\newcommand{\K}{\mathbb{K}}

\def\ov{\overline}

\def\B2star{\overline{B}_X^{w(X^{\ast\ast},X^{\ast})}}

\title{On the polynomial Lindenstrauss theorem}
\author{Daniel Carando %
\and Silvia Lassalle %
\and Martin Mazzitelli }

\thanks{This project was supported in part by UBACyT W746, UBACyT X218  and CONICET PIP 0624}

\address{Departamento de Matem\'{a}tica - Pab I,
Facultad de Cs. Exactas y Naturales, Universidad de Buenos Aires,
(1428) Buenos Aires, Argentina and IMAS-CONICET}

\email{dcarando@dm.uba.ar, slassall@dm.uba.ar, mmazzite@dm.uba.ar}


\date{}

\begin{document}
\baselineskip=.65cm

\begin{abstract}
Under certain hypotheses on the Banach space  $X$, we show that the set of $N$-homogeneous polynomials from $X$ to any dual space, whose Aron-Berner extensions are norm attaining, is dense in the space of all continuous $N$-homogeneous polynomials. To this end we prove an integral formula for the duality between
tensor products and polynomials. We also exhibit examples of Lorentz sequence spaces  for which there is no polynomial Bishop-Phelps theorem, but our results apply. Finally we address quantitative versions,  in the sense of Bollob\'as, of these results.
\end{abstract}

\maketitle

\section*{Introduction}
The Bishop-Phelps theorem \cite{BishPhel61} states that for any
Banach space $X$, the set of norm attaining bounded linear functionals is dense in $X'$, the dual space of $X$. Since the appearance of this result in 1961,
the study of norm attaining functions has attracted the attention of many authors.  Lindenstrauss showed
that there is no Bishop-Phelps theorem for linear bounded operators \cite{Lind}. Nevertheless, he proved that the set of bounded linear operators (between any
two Banach spaces $X$ and $Y$) whose second adjoints attain their
norm, is dense in the space of all operators.  This result was later extended for multilinear operators by Acosta, Garc\'ia and Maestre \cite{AGM}.  These kinds of results are referred to as  Lindenstrauss type theorems.  It is worth mentioning that a Bishop-Phelps theorem does not hold in general even for scalar-valued bilinear forms \cite{AAP}. Moreover, Choi showed in \cite{Choi97} that there is no Bishop-Phelps theorem for scalar-valued bilinear forms on $L_1[0,1]\times L_1[0,1]$. On the other hand, Finet and Pay\'{a} ~\cite{FinPay98} proved a Bishop-Phelps theorem for operators from  $L_1[0,1]$ to $L_\infty[0,1]$. As a consequence, we see that positive results for operators from a Banach space $X$ to its dual $X'$, do not imply positive results for bilinear forms on $X\times X$.

In the context of homogeneous polynomials, where there is no Bishop-Phelps theorem either,  the symmetric structure presents an additional difficulty. In \cite{ArGM} Aron, Garc\'\i a and Maestre showed  a polynomial Lindenstrauss theorem for the case of scalar-valued $2$-homogeneous polynomials. This was extended to vector-valued 2-homogeneous polynomials by Choi, Lee and Song \cite{ChoLeeSon10}. The aim of this work is to show a polynomial Lindenstrauss theorem for arbitrary degrees.

To achieve our goal, we first present an integral representation formula for the duality between
tensor products and polynomials. Namely, if  $X$ is a Banach spaces whose dual is separable and has the approximation property, we see in Theorem~\ref{formula integral vectorial} that any element in the tensor product  $ (\tilde{\otimes}_{\pi_s}^{N,s}X) \tilde{\otimes}_{\pi} Y$ is associated to a  regular Borel measure on $(B_{X''}, w^*) \times (B_{Y''}, w^*)$, for any Banach space $Y$. This integral formula somehow extends those given in \cite{GarGreMae09, GreRya06}.  In Theorem~\ref{Lind polinomial}, we apply our integral representation to prove a Lindenstrauss theorem for homogeneous polynomials from Banach spaces $X$ satisfying the hypotheses above,  into any dual space (and, therefore, for scalar-valued homogeneous polynomials on $X$). For instance, our result is valid for Banach spaces $X$ with shrinking bases.
Preduals of Lorentz sequence spaces $d_*(w,1)$  (see Section~\ref{ejemplos en preduales de lorentz})  are typical examples of spaces in which there is no polynomial Bishop-Phelps theorem. Nevertheless, our polynomial Lindenstrauss theorem applies  since they have shrinking bases.
In particular, those spaces with $w\in\ell_2$ do not satisfy the scalar-valued polynomial Bishop-Phelps theorem for any degree $N\ge 2$, but satisfy the polynomial Lindenstrauss theorem for \emph{every} degree  (see Example~\ref{primer ejemplo}). Moreover, for many  admissible sequences $w$,  we show that there exists some $1 < r < \infty$ such that the same happens for $\ell_r$-valued polynomials on $d_*(w,1)$ of any degree $N\ge 1$ (see Proposition~\ref{no vale BP} and the subsequent comments).

Bollob\'as  \cite{Boll} showed a quantitative version of the Bishop-Phelps theorem (see Section~\ref{no vale lindenstrauss-bollobas} for details). It seems natural to wonder about the validity of the corresponding quantitative versions of Lindenstrauss  type theorems, which we  call Lindenstrauss-Bollob\'{a}s theorems. For linear operators, it is shown \cite[Example~6.3]{AAGM} that no such result holds in general.
We  see that there is no Lindenstrauss-Bollob\'as theorem in the (scalar or vector-valued)  multilinear and polynomial settings (see Propositions~\ref{no vale LB multilineal} and~\ref{no vale LB polinomial}). We remark that the bilinear scalar-valued case is not a mere translation of the counterexample exhibited in \cite{AAGM} for operators. Here, the authors follow the ideas of~\cite{Lind} to show that the theorem fails for the identity map from $X=c_0$ to $Y$, a renorming of $c_0$ such that $Y''$ is strictly convex. This example cannot be modified to obtain a counterexample for bilinear mappings. Also, our construction provides a new counterexample for the Lindenstrauss-Bollob\'{a}s theorem for operators.

For further reading on polynomials and multlinear mappings on infinite dimensional Banach spaces we refer the reader to \cite{Din} and \cite{Muj}. An excellent survey on denseness of norm attaining  mappings can be found in \cite{Aco06}, see also the references therein.

\section{Preliminaries} \label{preliminares}

 Throughout this paper $X$ denotes a Banach space, while $X'$,  $B_X$ and $S_X$ denote respectively the
topological dual, the closed unit ball and the unit sphere of $X$. For $X_1,\ldots,X_N$ and $Y$
Banach spaces,  $\mathcal{L}(X_1,\dots,X_N;Y)$ stands for the space of continuous $N$-linear
maps $\Phi\colon X_1 \times \cdots\times X_N \rightarrow Y$ endowed with the supremum norm
$$\Vert \Phi\Vert = \sup \{\Vert \Phi(x_1,\dots,x_N)\Vert: \quad x_j \in B_{X_j}, \quad 1\leq j \leq N\}.$$
If $X_1=\cdots=X_N=X$ we simply write $\mathcal{L}(^NX;Y)$.
A function $P\colon X\to Y$  is said to be a (continuous) $N$-homogeneous
polynomial if there is a (continuous) $N$-linear map $$\Phi\colon
X\times \overset{N}{\cdots}\times X \to Y$$ such that $P(x)= \Phi(x,\ldots,x)$ for all $x\in X$.
We denote by ${\mathcal P}(^NX;Y)$ the Banach space of all continuous
$N$-homogeneous polynomials from $X$ to $Y$ endowed with the supremum norm
$$\Vert P\Vert = \sup_{x \in B_X} \Vert Px\Vert.$$
A  polynomial $P$ in ${\mathcal P}(^NX;Y)$ is said to be of
finite type if there exist $\{x'_j\}_{j=1}^m$ in $X'$ and  $\{y_j\}_{j=1}^m$ in $Y$ such that
$P(x)=\sum_{j=1}^m  x'_j(x)^Ny_j$ for all $x$ in $X$.   The subspace of all finite type $N$-homogeneous polynomials  is denoted by $\mathcal{P}_f(^NX;Y)$. When $Y=\K$ is the scalar field,  $\K=\R$ or $\C$, we omit it and write for instance $\mathcal{L}(^NX)$, ${\mathcal P}(^NX)$ or  $\mathcal{P}_f(^NX)$.

We say that a linear operator $T \in \mathcal{L}(X;Y)$  attains its norm (or is norm attaining) if there exists $a \in B_{X}$ such that $\| T(a)\| = \Vert T \Vert$.
Also, a multilinear operator $\Phi$ attains its norm if there exists a $N$-tuple $(a_1, \dots, a_N) \in B_{X_1} \times \cdots\times B_{X_N}$ such that $\| \Phi(a_1,\dots,a_N)\| = \| \Phi\|$. Analogously, $P\in  \mathcal{P}(^NX;Y)$ attains its norm if there exists $a \in B_{X}$ such that $\|P(a)\| = \Vert P \Vert$.
When it is opportune we write   $\mathcal{NAP}(^NX;Y)$ to denote the set of all norm attaining $N$-homogeneous polynomials of $\mathcal{P}(^NX;Y)$.

Polynomials in $\mathcal P(^NX)$ can be considered as continuous linear functionals on the symmetric projective tensor product as follows.
Given a  symmetric tensor $u$ in $\otimes^{N,s} X$,  the symmetric projective norm $\pi_s$ of
$u$ is defined to be
$$
\pi_s(u)=\inf\Big\{\sum_{j=1}^m |\lambda_j| \|x_j\|^N\colon u=\sum_{j=1}^m\lambda_j x_j^N, (\lambda_j)_{j=1}^m \subset \K, (x_j)_{j=1}^m\subset X\Big\}.
$$
We denote the completion of ${\otimes^{N,s}} X$ with respect to $\pi_s$ by  $\tilde{\otimes}_{\pi_s}^{N,s}X$. Then,
$$
\mathcal{P}(^NX) = (\tilde{\otimes}_{\pi_s}^{N,s}X)'
$$
isometrically, where the identification is given by the duality
$$
L_P(u):=\left\langle u,P \right\rangle = \sum_{j=1}^\infty \lambda_j P(x_j),
$$
for $P \in \mathcal{P}(^NX)$ and $u \in \tilde{\otimes}_{\pi_s}^{N,s}X$,  $u = \sum_{j=1}^\infty \lambda_j x_j^N$. Also, for polynomials with values in a dual space $Y'$ we have the isometric isomorphism
\begin{equation}\label{dualidad tensores}
\mathcal{P}(^NX;Y') = \left( (\tilde{\otimes}_{\pi_s}^{N,s}X) \tilde{\otimes}_{\pi} Y \right)'.
\end{equation} Here the duality  is given by
\begin{equation}\label{classical duality}
L_P(u):=\left\langle u,P \right\rangle = \sum_{k=1}^\infty  \sum_{j=1}^\infty \lambda_{k,j} P(x_{k,j}) (y_k)
\end{equation}
for any $P \in \mathcal{P}(^NX;Y')$ and $u= \sum_{k=1}^\infty v_k \otimes y_k$, where $(y_k)_k \subset Y$ and $(v_k)_k \subset \tilde{\otimes}_{\pi_s}^{N,s}X$, with $v_k =\sum_{j=1}^\infty \lambda_{k,j} x_{k,j}^N$ for all $k$.

Recall that the canonical (Arens) extension of a multilinear function is obtained by weak-star density as follows (see \cite{Arens} and \cite[1.9]{DefFlo93}).
Given $\Phi \in \mathcal{L}(X_1,\dots,X_N;Y)$, the mapping $\overline{\Phi} : X''_1 \times\cdots\times X''_N \longrightarrow Y''$ is defined by
\begin{eqnarray}\label{arens extension}
\overline{\Phi}(x_1'',\ldots,x_N'') = w^* - \lim_{\alpha_1}\ldots\lim_{\alpha_N} \varphi(x_{1,{\alpha_1}},\ldots,x_{N,{\alpha_N}})
\end{eqnarray}
where $(x_{j,{\alpha_j}})_{\alpha_j } \subseteq X$ is a net  $w^*$-convergent to  $x''_j\in X''_j$, $j=1,\ldots, N$. For $N=1$ this recovers the definition of the bitranspose of a continuous operator.

The Aron-Berner extension \cite{AB} of a polynomial  $P \in \mathcal{P}(^NX;Y)$  is the polynomial   $\ov P \in \mathcal{P}(^NX'';Y'')$,  defined by
$\overline{P}(x'')= \overline{\Phi}(x'',\ldots ,x'')$, where $\Phi$ is the unique symmetric $N$-linear mapping associated to $P$.
We also have $\|\overline P\|=\|P\|$, see \cite{DG}.

\section{Integral representation of tensors and the polynomial Lindenstrauss theorem} \label{formula integral} \label{lindenstrauss polinomial}

As a consequence of the principle of local reflexivity, given a Banach space   $X$ whose dual $X'$ is separable and enjoys the approximation property, it is possible to find a sequence of finite rank operators $(T_n)_n$ on  $X$ such that
both  $T_n\longrightarrow Id_X$ and $T'_n\longrightarrow Id_{X'}$ in the strong operator topology \cite[p.288-289]{HandbookI}. In fact, the existence of such a sequence
is actually equivalent to  $X'$ being separable with the approximation property. Clearly, we also have
$\sup_n \Vert T_n\Vert < \infty$,
\begin{equation}\label{schauderapprox} T_n{''}(X'') \subseteq J_X(X)\quad\text{ and }\quad T_n{''}(x'') \xrightarrow[n \rightarrow \infty]{w^*}\,x'' \text{
for all } x'' \in X'',\end{equation}where $J_X:X\to X''$ is the canonical inclusion.
\begin{lemma} \label{aproximacion puntual por pol. tipo finito}
Let $X, Y$ be Banach spaces and suppose that  $X'$ is separable and has the approximation property. Then, for each polynomial $P \in \mathcal{P}(^NX;Y')$ there exists a norm-bounded multi-indexed sequence of finite type polynomials $$(P_{n_1,\dots,n_N})_{(n_1,\dots,n_N)\in\N^N}\subset \mathcal{P}_f(^NX;Y')$$ such that the Aron-Berner extension of $P$ is given by the  iterated limit
\begin{equation} \label{limite de medibles}
\overline{P}(x'')(y'') = \lim_{n_1 \rightarrow \infty} \ldots \lim_{n_N \rightarrow \infty} \overline{P_{n_1,\ldots,n_N}}(x'')(y''),
\end{equation}
for each $x'' \in X''$ and $y'' \in Y''$.
\end{lemma}

\begin{proof}
Consider  a sequence of finite rank operators $(T_n)_n$ on  $X$ such that
both  $T_n$ and $T'_n$ converge to the respective identities in the strong operator topology. Let
$\Phi$ be the symmetric $N$-linear form associated to $P$ and fix $x'' \in X''$. Combining \eqref{arens extension} with \eqref{schauderapprox} we can compute the Aron-Berner extension of $P$ as
\begin{eqnarray*}
\overline{P}(x'') &=& w^* - \lim_{n_1 \rightarrow \infty} \dots \lim_{n_N \rightarrow \infty} \overline{\Phi}(T_{n_1}''(x''),\dots,T_{n_N}''(x'')) \\
&=& w^* - \lim_{n_1 \rightarrow \infty} \ldots \lim_{n_N \rightarrow \infty} \overline{\Phi \circ (T_{n_1},\ldots,T_{n_N})}(x'',\ldots,x'').
\end{eqnarray*}
The result now follows taking, for each $(n_1,\dots,n_N)\in\N^N$, the homogeneous polynomial
$P_{n_1,\ldots,n_N}\colon X \longrightarrow Y'$  given by
$P_{n_1,\ldots,n_N} = \Phi \circ (T_{n_1},\ldots,T_{n_N})$.
\end{proof}

Now we prove the integral representation for the elements in the tensor product $(\tilde{\otimes}_{\pi_s}^{N,s}X) \tilde{\otimes}_{\pi} Y$, which should be compared with \cite[Theorem~1]{GreRya06} and \cite[Remark~3.6]{GarGreMae09}. As usual, we consider $B_{X''}$ and $B_{Y''}$ endowed with their weak-star topologies, which make them compact sets.

\begin{theorem} \label{formula integral vectorial}
Let $X, Y$ be Banach spaces and suppose that  $X'$ is separable and has the approximation property. Then, for each $u \in (\tilde{\otimes}_{\pi_s}^{N,s}X) \tilde{\otimes}_{\pi} Y$ there exists a regular Borel measure $\mu_u$ on $(B_{X''}, w^*) \times (B_{Y''}, w^*)$ such that $\Vert \mu_u\Vert \leq \Vert u\Vert_{\pi}$ and

\begin{equation} \label{ec. formula integral vectorial}
\left\langle u,P \right\rangle = \int_{B_{X''}\times B_{Y''}} \overline{P}(x'')(y'') d\mu_u(x'',y''),
\end{equation}
for all $P \in \mathcal{P}(^NX;Y')$.
\end{theorem}

\begin{proof}
We first prove the formula for finite type polynomials.  Finite type polynomials from $X$ to $Y'$ can be seen as an isometric subspace of $C(B_{X''} \times B_{Y''})$, identifying   a polynomial  $P=\sum_{j=1}^m (x_j')^N(\cdot) y_j'$ with the function
\begin{equation}\label{eq tipo finito cont de la bola}(x'',y'')\mapsto \sum_{i=1}^m x''(x_j')^N\  y''(y_j')=\overline P(x'')(y'').\end{equation}
On the other hand, from  duality \eqref{dualidad tensores} we have isometrically $$
(\tilde{\otimes}_{\pi_s}^{N,s}X) \tilde{\otimes}_{\pi} Y \hookrightarrow
\left( (\tilde{\otimes}_{\pi_s}^{N,s}X) \tilde{\otimes}_{\pi} Y \right)''=
\left(\mathcal{P}(^NX;Y')   \right)'.$$ Therefore, each $u \in (\tilde{\otimes}_{\pi_s}^{N,s}X) \tilde{\otimes}_{\pi} Y$ defines a linear functional on $\mathcal{P}(^NX;Y')$ which can be restricted  to a linear functional $\Lambda_u$  on the space of finite type polynomials. Note that $$\Lambda_u(P)=\langle u, P\rangle$$ for $P\in \mathcal{P}_f(^NX;Y')$ and that $\|\Lambda_u\|\le \|u\|_{\pi_s}$.
Since $\mathcal{P}_f(^NX;Y')$ is a subspace of $C(B_{X''} \times B_{Y''})$, we extend $\Lambda_u$ by the Hahn-Banach theorem to
a continuous linear  functional on $C(B_{X''} \times B_{Y''})$ preserving the norm. Now, by the Riesz representation theorem, there is a regular Borel measure $\mu_u$ on $(B_{X''}, w^*) \times (B_{Y''},w^*)$ such that $\Vert \mu_u\Vert \leq \Vert u\Vert_{\pi}$ and
$$
\Lambda_u (f) = \int_{B_{X''} \times B_{Y''}} f(x'',y'') d\mu_u(x'',y'')
$$
for $f\in C(B_{X''} \times B_{Y''})$, where we still use $\Lambda_u$ for its extension to  $C(B_{X''} \times B_{Y''})$. In particular, we can consider $f= \sum_{i=1}^m (x_j')^N \otimes y_j' $ and its identification \eqref{eq tipo finito cont de la bola}, so we obtain the integral formula \eqref{ec. formula integral vectorial} for finite type polynomials.

Now, take $P \in \mathcal{P}(^NX;Y')$. By Lemma~\ref{aproximacion puntual por pol. tipo finito}, there exists a norm bounded multi-indexed sequence  of  finite type polynomials $(P_{n_1,\dots,n_N})_{(n_1,\dots,n_N)\in\N^N}$ satisfying equation~\eqref{limite de medibles}.
Since we have already proved the integral formula for finite type polynomials we have
\begin{equation*}
\left\langle u,P_{n_1,\ldots,n_N} \right\rangle = \int_{B_{X''} \times B_{Y''}} \overline{P_{n_1,\ldots,n_N}}(x'')(y'') d\mu_u(x'',y''),
\end{equation*}
for all $(n_1,\ldots,n_N) \in \N^N$. As the sequence $(P_{n_1,\dots,n_N})_{(n_1,\dots,n_N)\in\N^N}$ is norm bounded, we may apply $N$-times the bounded convergence theorem to obtain
\begin{eqnarray*}
\lim_{n_1 \rightarrow \infty} \dots \lim_{n_N \rightarrow \infty} \left\langle u,P_{n_1,\ldots,n_N} \right\rangle &=& \lim_{n_1 \rightarrow \infty} \ldots \lim_{n_N \rightarrow \infty} \int_{B_{X''} \times B_{Y''}} \overline{P_{n_1,\ldots,n_N}}(x'')(y'') d\mu_u(x'',y'') \\
&=& \int_{B_{X''} \times B_{Y''}} \overline{P}(x'')(y'') d\mu_u(x'',y'').
\end{eqnarray*}

It remains to show that $\left\langle u,P \right\rangle  = \lim_{n_1 \rightarrow \infty} \ldots \lim_{n_N \rightarrow \infty} \left\langle u,P_{n_1,\ldots,n_N} \right\rangle$. This follows from the fact that $\left\langle \ \cdot \ ,P \right\rangle$ and $\lim_{n_1 \rightarrow \infty} \ldots \lim_{n_N \rightarrow \infty} \left\langle \  \cdot  \ ,P_{n_1,\ldots,n_N} \right\rangle$ are linear continuous functions on $(\tilde{\otimes}_{\pi_s}^{N,s}X) \tilde{\otimes}_{\pi} Y$ which coincide on elementary tensors. Hence, the proof is complete.
\end{proof}

We are now ready to prove our  Lindenstrauss  theorem for homogeneous polynomials.

\begin{theorem} \label{Lind polinomial}
Let $X, $ $Y$ be Banach spaces. Suppose that $X'$ is separable and has the approximation property. Then, the set of all polynomials in $\mathcal{P}(^NX;Y')$ whose Aron-Berner extension attain their norm is dense in $\mathcal{P}(^NX;Y')$.
\end{theorem}

\begin{proof}
Given  $Q \in \mathcal{P}(^NX;Y')$  consider its associated linear function $L_Q \in \left( (\tilde{\otimes}_{\pi_s}^{N,s}X) \tilde{\otimes}_{\pi} Y \right)'$, defined as in  \eqref{classical duality}.  The Bishop-Phelps theorem asserts that, for $\varepsilon >0$ there exists a norm attaining functional $L=L_P$ such that $\Vert L_Q - L_P \Vert < \varepsilon$, for $P$ some polynomial in $\mathcal{P}(^NX;Y')$.
Since $\Vert L_Q - L_P \Vert = \Vert Q - P \Vert$, once we prove that $\overline{P}$ is norm attaining the result follows.

We take $u \in  (\tilde{\otimes}_{\pi_s}^{N,s}X) \tilde{\otimes}_{\pi} Y $  such that $\Vert u\Vert_{\pi}=1$ and $\vert L_P(u)\vert = \Vert L_P\Vert = \Vert P\Vert$. By Theorem~\ref{formula integral vectorial}, there exists a regular Borel measure $\mu_{u}$ on $B_{X''}$ such that
\begin{equation}
\left\langle u,P \right\rangle = \int_{B_{X''} \times B_{Y''}} \overline{P}(x'')(y'') d\mu_{u}(x'',y'') \quad \text{and} \quad \Vert \mu_{u}\Vert \leq \Vert u\Vert_{\pi}=1.
\end{equation}
Then,
\begin{eqnarray*}
\Vert P\Vert = \vert L_P(u)\vert \leq \int_{B_{X''} \times B_{Y''}} \vert\overline{P}(x'')(y'')\vert\ d\vert\mu_{u}\vert(x'',y'') \leq \Vert \overline{P}\Vert \Vert \mu_{u}\Vert \leq \Vert P\Vert.
\end{eqnarray*}
In consequence,
\begin{equation}
\Vert P\Vert= \int_{B_{X''} \times B_{Y''}} \vert\overline{P}(x'')(y'')\vert\ d\vert\mu_{u}\vert(x'',y'').
\end{equation}
In particular, $\Vert \mu_u\Vert=1$  and $\vert\overline{P}(x'')(y'')\vert = \Vert P\Vert$ almost everywhere (for $\mu_{u}$). Hence $\overline{P}$ attains its norm.
\end{proof}

Banach spaces with shrinking bases satisfy the hypotheses of the theorem. Examples of non-reflexive Banach spaces with shrinking bases are preduals of Lorentz sequence spaces, which will be treated in the next section. As an immediate consequence, we  state the scalar version of the previous result.

\begin{corollary}\label{Lind polinomial escalar}
Let $X$ be a  Banach space whose dual is separable and has the approximation property. The set of all polynomials in $\mathcal{P}(^NX)$ whose Aron-Berner extension attains  the norm is dense in $\mathcal{P}(^NX)$.
\end{corollary}

It should be noted that there are Banach spaces which do not satisfy the hypotheses of Theorem~\ref{Lind polinomial} and Corollary~\ref{Lind polinomial escalar}, for which the polynomial Lindenstrauss theorem holds. For example, Theorem 2.7 in \cite{ChoiKim96} states that, if $X$ has the Radon-Nikod\'{y}m property, then the set of norm attaining polynomials from $X$ to any Banach space $Y$ is dense in $\mathcal P (^NX,Y)$, for any $N\in \mathbb N$. As a consequence, $\ell_1$ (whose dual is not separable) satisfies a polynomial Bishop-Phelps theorem, which is stronger than the polynomial Lindenstrauss theorem.

\section{Examples on preduals of Lorentz sequence spaces} \label{ejemplos en preduales de lorentz}
Lorentz sequence spaces  and their preduals are classical Banach spaces that proved  useful to get a better understanding of some problems related to norm attaining operators and nonlinear functions. In fact, it is  $d_*(w, 1)$, a predual of a Lorentz sequence space $d(w,1)$, on which the first counterexample to the Bishop-Phelps theorem for  bilinear forms and 2-homogeneous scalar-valued polynomials \cite{AAP} was modeled. Moreover,  the set of $N$-homogeneous polynomials attaining  the norm is dense in $\mathcal P(^N d_*(w,1))$ if and only if the weight $w$ is not in $\ell_N$, \cite[Theorem~3.2]{JP}. Also, there is an analogous result for multilinear forms \cite[Theorem~2.6]{JP}.

We recall the definition and some elementary facts about Lorentz sequence spaces (see \cite[Chapter~4.e]{LT1} for further details). Let $w=(w_i)_{i \in \mathbb{N}}$ be a decreasing sequence of nonnegative real numbers with $w_1 = 1$, $\lim w_i =0$ and $\sum_{i}w_i = \infty$. Such sequences are called \emph{admissible}. If  $1 \leq s < \infty$ is fixed, the Lorentz sequence space $d(w,s)$ associated to an admissible sequence $w=(w_i)_{i\in \N}$ is the vector space of all bounded sequences $x=(x(i))_i$ such that
$$
\Vert x\Vert_{w,s} : =\left(\sum_{i=1}^{\infty} x^*(i)^s w_i\right)^{1/s} < \infty,
$$
where $x^*=(x^*(i))_i$ is the decreasing rearrangement of $(x(i))_i$.  The norm $\Vert \cdot\Vert_{w,s}$ makes $d(w,s)$  a Banach space which is reflexive if and only if $1<s<\infty$.

For $s=1$, i.e., in the nonreflexive case, the dual  space of $d(w,1)$ is denoted by  $d^*(w,1)$ and consists of all bounded sequences $x$ such that
$$
\Vert x\Vert_W : = \sup_{n}\frac{\sum_{i=1}^n x^*(i)}{W(n)} < \infty,
$$
where $W(n)= \sum_{i=1}^n w_i$. The predual of the Lorentz space $d(w,1)$, denoted by $d_*(w,1)$,  is the subspace of $d^*(w,1)$ of all the sequences $x$ satisfying
$$
\lim_{n\to \infty}\frac{\sum_{i=1}^n x^*(i)}{W(n)}=0.
$$

If $X$ denotes any of the spaces $d_*(w,1), d(w,1)$, $d^*(w,1)$, the condition $w_1=1$  is equivalent to the assumption that $\|e_i\|=1$ for all $i$ in $\N$, where $e_i$ stands for the canonical $i$-th vector of $X$. For any admissible sequence $w$,  $X$ is contained in $c_0$ as a set and therefore, for each element $x\in X$  there exists an injective mapping $\sigma \colon \mathbb{N} \rightarrow \mathbb{N}$ such that $x^*$  is of the form $x^* = (\vert x(\sigma(i))\vert)_i$.

If $w \in \ell_r$, $1 < r< \infty$, a direct application of H\"older's inequality shows that the canonical inclusion  $\ell_{r^*}\hookrightarrow d(w,1)$ is a bounded operator. By transposition and restriction, both mappings \begin{equation}\label{inclusiones}d^*(w,1) \hookrightarrow \ell_r\quad\text{ and }\quad d_*(w,1) \hookrightarrow \ell_r\end{equation} are also bounded.
The geometry of the unit ball of $d_*(w,1)$ (more precisely the lack of extreme points) plays a crucial role in the proof of \cite[Theorem~2.6]{JP} and \cite[Theorem~3.2]{JP}, and also in our results below. The fundamental property of these spaces \cite[Lemma~2.2]{JP} is that
 any $x \in B_{d_*(w,1)}$, satisfies the following condition:
\begin{equation}\label{propiedad extremal}
\exists \  n_0 \in \mathbb{N} \  \text{\ and\ }\ \delta >0\  \text{\ such that\ }\quad  \Vert x + \lambda e_n\Vert_W \leq 1,\quad  \forall \ |\lambda| \le  \delta \text{\  and \  }  n \geq n_0.
\end{equation}

Finally, preduals of Lorentz sequence spaces have shrinking basis.
Then, at the light of Corollary~\ref{Lind polinomial escalar}  and \cite[Theorem~3.2]{JP}, for $w \in \ell_N$,  the polynomial Lindenstrauss theorem holds for $d_*(w,1)$ but the polynomial Bishop-Phelps theorem does not.

\begin{example} \label{primer ejemplo}
Let  $w$  be an admissible sequence  in $\ell_M$,  $M\ge 2$.
Then,  the set of norm-attaining $N$-homogeneous polynomials on $d_*(w,1)$ is not dense in $\mathcal{P}(^Nd_*(w,1))$ for any $N\ge M$, while the set of those whose Aron-Berner extension attains the norm is dense for every $N$.
\end{example}

In particular, the above example shows that there exists a Banach space satisfying the polynomial Lindenstrauss theorem and failing the scalar-valued polynomial Bishop-Phelps theorem  for $N$-homogeneous polynomials, for all $N$ (just take $w\in \ell_2$). Now we show that, given any admissible sequence $w \in \ell_r$, $1 <r<\infty$, there exists a Banach space $Y$ such that the set of norm attaining $N$-homogeneous polynomials  fails to be dense in $\mathcal{P}(^Nd_*(w,1);Y')$ for all $N\ge 2$ (while the polynomial Lindenstrauss theorem holds for any $N$). In order to do so, we state as lemmas two useful results. The proof of the first one is similar to those of  \cite[Lemma~3.1]{JP}  and \cite[Proposition~4]{Lind}.

\begin{lemma} \label{lema extremales poli}
In the complex case, let $X$  be a Banach sequence space and  $Y$ be strictly convex.
If $P \in \mathcal{P}(^NX;Y)$ attains the norm at some point satisfying condition~\eqref{propiedad extremal} for some $n_0\in\N$, then $P(e_n)=0$, for all $n \geq n_0$.
\end{lemma}

As we have already mentioned, condition~\eqref{propiedad extremal} is satisfied by any $a$ in the unit ball of  $d_*(w,1)$. Therefore, the previous lemma applies to  every norm attaining polynomial from $d_*(w,1)$ to any strictly convex Banach space. The proof of the following result
can be extracted  from \cite[Theorem~3.2]{JP}.

\begin{lemma} \label{ppio mod max real}
For the  real case, let $w$ be an admissible sequence in $\ell_N$, $N\ge 2$, and take $M$ the smallest natural number such that $w \in \ell_M$. Suppose that $p \in \mathcal{P}(^N d_*(w,1))$ attains its norm at $a \in B_{d_*(w,1)}$ and let $\phi$ be the symmetric $N$-linear form associated to $p$.
\begin{enumerate}
\item[(i)] If $p(a) >0$\quad  then \quad  $\limsup_n \phi(a,\ldots,a,e_n,\overset{M}{\dots},e_n) \leq 0$.
\item[(ii)] If $p(a) <0$\quad then \quad $\liminf_n \phi(a,\ldots,a,e_n,\overset{M}{\dots},e_n) \geq 0$.
\end{enumerate}
\end{lemma}

\begin{proposition} \label{no vale BP}
Let $w$ be an admissible sequence and suppose $M$ is the smallest natural number such that $w\in \ell_M$ (we assume such an $M$ exists).  Then,  $\mathcal{NAP}(^Nd_*(w,1); \ell_M)$ is not dense in $\mathcal{P}(^Nd_*(w,1); \ell_M)$, for any $N \in \mathbb{N}$.
\end{proposition}
\begin{proof}
\textit{The complex case}.
Since $w \in \ell_M$, we consider   $Q\colon d_*(w,1) \longrightarrow \ell_M$ given by $Q(x) = (x(i)^N)_i$, which is a well defined and continuous polynomial by \eqref{inclusiones}. Suppose that $Q$ is approximable by norm attaining polynomials. Thus, for fixed $0<\varepsilon <1$ there exists $P \in \mathcal{NAP}(^Nd_*(w,1);\ell_M)$ such that $\Vert P-Q\Vert<\varepsilon$ and therefore $\vert \Vert P(e_n)\Vert - \Vert Q(e_n)\Vert\vert < \varepsilon$ for all $n \in \mathbb{N}$. By Lemma~\ref{lema extremales poli}, there exists $n_0$ such that $P(e_n)=0$ for all $n \geq n_0$. Hence, $1 = \Vert Q(e_n) \Vert < \varepsilon$, for all $n \geq n_0$, and  the result follows by contradiction.

\textit{The real case}. Now, we consider  $Q\colon d_*(w,1) \longrightarrow \ell_M$ the  continuous polynomial defined by $Q(x) = (x(1)^{N-1} x(i))_i$. Suppose that $Q$ is approximable by norm attaining polynomials and fix $\varepsilon>0$. Norm one $M$-homogeneous polynomials (on $ \ell_M$) are uniformly equicontinuous. Therefore, we can take $P \in \mathcal{NAP}(^Nd_*(w,1);\ell_M)$ close enough to $Q$ such that
\begin{equation}\label{eq-equicont}
\Vert q \circ Q - q \circ P\Vert < \varepsilon\  \textstyle{\frac{(NM)!}{(NM)^{NM}}}
\end{equation}for every norm one polynomial $q\in\mathcal P(^M\ell_M)$.

Let $a \in B_{d_*(w,1)}$ be such that $\Vert P(a)\Vert= \Vert P\Vert$ and consider the norm one $M$-homogeneous polynomial $q_{P,a}\colon \ell_M \longrightarrow \mathbb{R}$ given by
$$
q_{P,a}(x) = \sum_i \lambda_i^M\ x(i)^M,
$$
where $ \lambda_i=1$ if $P(a)(i)\ge 0$ and $ \lambda_i= -1$ otherwise.
Note that $q_{P,a} \circ P \colon d_*(w,1) \longrightarrow \mathbb{R}$ is an $NM$-homogeneous polynomial attaining its norm at $a$, with $q_{P,a} \circ P (a) =\|P\|^M$. Also,
$$
q_{P,a} \circ Q(x) = x(1)^{M(N-1)} \sum_i  \lambda_i^M\ x(i)^M, \quad \text{for all}\quad x\in d_*(w,1),
$$
and then
$\|q_{P,a} \circ Q\|=\|Q\|^M$. Let $\phi$ and $\psi$ be the symmetric $NM$-linear forms associated to $q_{P,a} \circ P$ and $q_{P,a} \circ Q$, respectively. By \eqref{eq-equicont}, we have
$\|\phi -\psi\|< \varepsilon$.

Since $q_{P,a} \circ P (a) >0$,  Lemma~\ref{ppio mod max real} gives that
\begin{equation} \label{NMlineal phi}
\limsup_n \phi(a,\dots,a,e_n,\overset{M}{\dots},e_n) \leq 0.
\end{equation}
On the other hand,
\begin{equation} \label{NMlineal psi}
{\textstyle \binom{NM}{M}} \psi(a,\dots,a,e_n,\overset{M}{\dots},e_n) =  \lambda_n^M a(1)^{M(N-1)}.
\end{equation}
Suppose that $M$ is even. Since $\|\phi -\psi \| < \varepsilon$,  combining \eqref{NMlineal phi} and \eqref{NMlineal psi} we obtain
\begin{equation} \label{desigualdad con a1}
 \vert a(1)\vert^{M(N-1)} = \textstyle{ \binom{NM}{M}} \lim_n \psi(a,\dots,a,e_n,\overset{M}{\dots},e_n) \leq  \binom{NM}{M} \varepsilon.
\end{equation}

Therefore,
\begin{eqnarray*}
\|Q\|^M=\Vert q_{P,a} \circ Q\Vert &\leq& \Vert q_{P,a} \circ P\Vert + \varepsilon\ \textstyle{\frac{(NM)!}{(NM)^{NM}}} \\ & < & \vert q_{P,a} \circ Q(a)\vert + 2\varepsilon \ \textstyle{\frac{(NM)!}{(NM)^{NM}}}\\
& \leq & \vert a(1)\vert^{M(N-1)} \sum_i \vert a(i)\vert^{M} + 2\varepsilon\ \\
&\leq& \varepsilon \Big( {\textstyle \binom{NM}{M}} \sum_i w_i^M + 2\Big).
\end{eqnarray*}
Since the last inequality is valid for all $\varepsilon >0$, we get that $\Vert Q\Vert=0$, which is a contradiction.

Now suppose that $M$ is odd. We give the proof for $N$ even, the remaining case being analogous. Without loss of generality, we may assume that  $a(1)\ge 0$.

Note that $q_{P,a} \circ P$ also attains its norm at $-a$ and $q_{P,a} \circ P(-a)=\|P\|^M$. By  Lemma~\ref{ppio mod max real},
\begin{equation*}
\limsup_n \phi(-a,\dots,-a,e_n,\overset{M}{\dots},e_n) \leq 0.
\end{equation*}
Then $\liminf_n \phi(a,\dots,a,e_n,\overset{M}{\dots},e_n) \ge 0$ and therefore,
by \eqref{NMlineal phi},
\begin{equation}\label{NM lineal psi da 0}
\lim_n  \phi(a,\dots,a,e_n,\overset{M}{\dots},e_n) = 0.
\end{equation}
If $\lambda_n=1$ for infinitely many $n$'s, using \eqref{NMlineal psi} and the limit above  we again obtain $|a(1)|^{M(N-1)} < \binom{NM}{M} \varepsilon$. Thus, we may proceed as in the previous case to get $\Vert Q\Vert=0$, a contradiction.

Suppose that $\lambda_n=1$ for only finitely many $n$'s. Since
$$
{\textstyle \binom{NM}{M}} \psi(-a,\dots,-a,e_n,\overset{M}{\dots},e_n) = - \lambda_n^M a(1)^{M(N-1)}.$$
we have
$$
{\textstyle \binom{NM}{M}} \limsup_n \psi(-a,\dots,-a,e_n,\overset{M}{\dots},e_n) = \vert a(1)\vert^{M(N-1)}.
$$
Together with \eqref{NM lineal psi da 0}, this implies that $|a(1)|^{M(N-1)} < \binom{NM}{M} \varepsilon$. Thus, we again derive that $\Vert Q\Vert=0$, whence the result follows by contradiction.
\end{proof}

The previous proposition shows  that, given any admissible sequence $w$ in $\ell_r$, $1<r<\infty$, there exists a dual space $Y'$ such that $\mathcal{P}(^Nd_*(w,1); Y')$ satisfies the polynomial Lindenstrauss theorem, but not the Bishop-Phelps theorem, for all $N \in \mathbb{N}$. This somehow extends Example~\ref{primer ejemplo}.

In the complex case, the proof of Proposition~\ref{no vale BP} works if we consider $\ell_r$ instead of $\ell_M$ for $M < r < \infty$. In other words,  $\mathcal{NAP}(^Nd_*(w,1); \ell_r)$ is not dense in $\mathcal{P}(^Nd_*(w,1); \ell_r)$, for any $N \in \mathbb{N}$ and  $M \le r < \infty$. Also, taking $Z$ a renorming of $c_0$ such that its bidual  is strictly convex, the polynomial $Q$ considered above is well defined from $d_*(w,1)$ to $Z''$ regardless of $w$ belonging to some $\ell_r$. In consequence,  $\mathcal{NAP}(^Nd_*(w,1); Z'')$ is not dense in $\mathcal{P}(^Nd_*(w,1); Z'')$, for any $N \in \mathbb{N}$ and  any admissible sequence $w$. On the other hand, the polynomial Lindenstrauss theorem holds in all these situations.

\section{On a quantitative version of the Lindenstrauss theorem} \label{no vale lindenstrauss-bollobas}

There is a quantitative version of the Bishop-Phelps theorem, due to Bollob\'as \cite{Boll} which states that, for any Banach space $X$, once we fix a linear functional $\varphi \in S_{X'}$ and $\tilde x  \in S_X$ such that $\varphi(\tilde x )$ is close enough to $1$,  it is possible to find a linear functional $\psi \in S_{X'}$ attaining its norm at some $a \in S_X$, such that, simultaneously, $\tilde x $ is close enough to $a$ and $\varphi$   is close enough to  $\psi$.

Suppose we have  Banach spaces $X$ and $Y$, for which the Bishop-Phelps theorem holds for $\mathcal L(X;Y)$, $\mathcal L(^N X;Y)$ or $\mathcal P(^N X;Y)$.  Is it possible to obtain a quantitative version of the theorem in any of these situations? This question was first posed and studied in the context of linear operators by Acosta, Aron, Garc\'{\i}a and Maestre  \cite{AAGM}. The authors show that a Bishop-Phelps-Bollob\'as theorem  holds  for $\mathcal L(\ell_1; Y)$ if and only if $Y$ has AHSP (approximate hyperplane series property). This last property is satisfied by  finite-dimensional Banach spaces, $L_1(\mu)$  for a $\sigma$-finite measure $\mu$, $C(K)$ spaces and uniformly convex Banach spaces. In particular, $\mathcal L(\ell_1, \ell_\infty)$ satisfies the Bishop-Phelps-Bollob\'as theorem. Continuing this line of research, Choi and Song extended the question  to the bilinear case \cite{ChoiSong}. Here the authors show that there is no Bishop-Phelps-Bollob\'as theorem for scalar bilinear forms on $\ell_1 \times \ell_1$, in contrast to the positive result for $\mathcal L(\ell_1, \ell_\infty)$. This should be compared to the already mentioned results of \cite{Choi97} and \cite{FinPay98}.

As stated in the introduction, the corresponding quantitative version of Lindenstrauss theorem for operators was addressed in \cite{AAGM}.  We devote this section to show that  there is no Lindenstrauss-Bollob\'{a}s theorem for multilinear mappings and polynomials on preduals of Lorentz sequence spaces. Before going on, some definitions are in order.

Following \cite{AAGM}, we say that the Bishop-Phelps-Bollob\'as theorem holds for $\mathcal{L}(X;Y)$, if for any $\varepsilon >0$ there are $\eta(\varepsilon), \beta(\varepsilon)>0$ (with $\beta(t) \xrightarrow[t \rightarrow 0]\,0$) such that given $T \in S_{\mathcal{L}(X;Y)}$ and $\tilde x  \in S_X$ with $\Vert T(\tilde x) \Vert > 1 - \eta(\varepsilon)$,  there exist $S \in S_{\mathcal{L}(X;Y)}$ and $a \in S_X$ satisfying:
$$
\Vert S(a)\Vert = 1, \quad \Vert a - \tilde x \Vert < \beta(\varepsilon) \quad \text{and} \quad \Vert S -T\Vert < \varepsilon.
$$

Regarding  Lindenstrauss-Bollob\'as type theorems, we say that the theorem holds for $\mathcal{L}(X;Y)$  if, with $\varepsilon, \eta$ and $\beta$ as above, given $T \in S_{\mathcal{L}(X;Y)}$ and  $\tilde x \in S_X$ with $\Vert T(\tilde x)\Vert > 1 - \eta(\varepsilon)$, there exist $S \in S_{\mathcal{L}(X;Y)}$ and $a \in S_{X''}$ satisfying:
$$
\Vert S'' (a)\Vert = 1,\ \Vert a - \tilde x \Vert < \beta(\varepsilon)\quad \text{and}\quad \Vert S -T\Vert < \varepsilon.
$$
More generally, we say that the Lindenstrauss-Bollob\'as theorem holds for $\mathcal{L}(^N X_1, \ldots, X_N; Y)$, if
given $\Phi \in \mathcal{L}(^N X_1, \ldots, X_N;Y)$ of norm 1, and $\tilde x_j \in S_{X_j}$, $j=1,\ldots, N$, with $\Vert \Phi(\tilde x_1,\ldots,\tilde x_N)\Vert > 1 - \eta(\varepsilon)$,  there exist $\Psi \in S_{\mathcal{L}(^N X_1, \ldots, X_N ;Y)}$ and $a_j \in S_{X_j''}$,  $j=1,\ldots, N$, satisfying that all the Arens extensions of $\Psi$ attain the norm at $(a_1,\ldots,a_N)$, $\Vert a_j - \tilde x_j\Vert < \beta(\varepsilon)$ for $j=1,\ldots,N$ and $\Vert \Phi -\Psi\Vert < \varepsilon$.
The polynomial version of the Lindenstrauss-Bollob\'{a}s theorem can be stated by analogy.

\begin{remark}\rm
Regarding the multilinear Lindenstrauss-Bollob\'{a}s theorem, we could have required a formally weaker condition on  $\Psi$: that merely \emph{one} of its Arens extensions attain its norm at $(a_1,\ldots,a_N)$. We do not know if the Lindenstrauss-Bollob\'{a}s theorem corresponding to this condition is equivalent to the former one. Anyway, we will see in Proposition~\ref{no vale LB multilineal} than even this weaker form of the theorem fails.
\end{remark}

As in Section \ref{ejemplos en preduales de lorentz}, a Banach space  with few extreme points will provide us with the proper environment to construct our counterexamples.  We state as a lemma the following  known result whose proof is similar to those of \cite[Lemma~2.2]{JP} and \cite[Proposition~4]{Lind}.

\begin{lemma}\label{lema extremales}
Let $X_1,\ldots,X_N$ be Banach sequence spaces and let $Y$ be a strictly convex Banach space.
If $\Phi \in \mathcal{L}(^N X_1, \ldots, X_N ;Y)$ attains its norm at $(a_1,\ldots,a_N) \in B_{X_1} \times \cdots\times B_{X_N}$ with $a_1,\ldots,a_N$ satisfying \eqref{propiedad extremal}, then there exists $n_0\in \N$ such that  $\Phi(e_{n_1},\ldots,e_{n_N}) =0$, for all $n_1,\ldots,n_N \geq n_0$.
\end{lemma}

Our next lemma shows that elements in $B_{d^*(w,1)}$ which are \textit{close} to elements in $B_{d_*(w,1)}$, satisfy condition~\eqref{propiedad extremal}.

\begin{lemma} \label{lema 2} Let $w$ be an admissible sequence.
Let $z \in B_{d^*(w,1)}$ and suppose there exists $x  \in d_*(w,1)$  such that  $\Vert z - x\Vert_W < \frac 12$. Then, $z$ satisfies \eqref{propiedad extremal}, that is:
there exist $\delta >0$ and $n_0 \in \mathbb{N}$ such that
$$
\Vert z + \lambda e_n\Vert_W \leq 1,  \quad \text{for all}\quad  |\lambda|\le \delta \quad\text{and all}\quad n \geq n_0.
$$
\end{lemma}

\begin{proof} If there exists $i \in \mathbb{N}$ so that $z^*(i)=0$ then $z\in d_*(w,1)$ and the result follows.  Then, we may suppose $z^*(i)> 0$, for all $i \in \mathbb{N}$.  Choose $\rho>0$ such that $\Vert z - x\Vert_W < \rho < \frac 12$.
Since $x \in d_*(w,1)$, there exists $n_1 \in \mathbb{N}$ such that, for all $n \geq n_1$,
$$
\frac{\sum_{i=1}^n x^*(i)}{W(n)} < \rho \quad \text{and} \quad \frac{\sum_{i=1}^n (z-x)^*(i)}{W(n)} < \rho.
$$
Then,
\begin{equation} \label{ec1 lema 2}
\sum_{i=1}^n z^*(i)
\leq \sum_{i=1}^n (z-x)^*(i) + \sum_{i=1}^n x^*(i) < 2 \rho W(n),
\end{equation}
for all $n \geq n_1$.

Let $n_2$ be the smallest natural number satisfying $n_2 > n_1$ and $z^*(n_2) < z^*(n_2-1)$. By \eqref{ec1 lema 2} and the choice of $\rho$ we may take $\delta>0$ such that
\begin{equation*}
z^*(n_2) + \delta < z^*(n_2-1) \qquad \text{and} \qquad \frac{\sum_{i=1}^n z^*(i) + \delta}{W(n)} < 1, \quad \text{ for all } n \geq n_1.
\end{equation*}
Let $\sigma\colon \mathbb{N} \rightarrow \mathbb{N}$ be an injective mapping satisfying $z^* = (\vert z(\sigma(i))\vert)_i$ and take $n_0 > \max\{\sigma(1),\ldots,\sigma(n_2)\}$. Let us show that $\Vert z + \lambda e_n\Vert_W \leq 1$ for all  $|\lambda|<\delta$ and all $n \geq n_0$. Note that if $n \geq n_0$, then $\vert z(n)\vert \leq z^*(n_2)$  and
\begin{equation*}
\vert z(n)  + \lambda \vert  < \vert z(n)\vert + \delta \leq z^*(n_2) + \delta < z^*(n_2-1) \leq z^*(n_1) \leq \dots \leq z^*(1).
\end{equation*}
If $m < n_2$, then
$$
\sum_{i=1}^m (z +\lambda e_n)^*(i) = \sum_{i=1}^m z^*(i) \le W(m).
$$
On the other hand, if  $m \ge n_2$,
$$
\sum_{i=1}^m (z +\lambda e_n)^*(i) \le \sum_{i=1}^m z^*(i) + \sum_{i=1}^m (\lambda e_n)^*(i) \le
\sum_{i=1}^m z^*(i) + \delta < W(m).
$$
Thus, the result follows.
\end{proof}

Now we are ready to show that there is no Lindenstrauss-Bollob\'{a}s theorem for multilinear forms or multilinear operators on preduals of Lorentz sequence spaces.

\begin{proposition} \label{no vale LB multilineal} Let $w$ be an admissible sequence  in $\ell_r$, for some $1<r<\infty$.  There is no Lindenstrauss-Bollob\'as theorem in the following cases:
\begin{enumerate}
 \item[\rm (a)]  for $\mathcal{L}(^Nd_*(w,1))$, if  $N \geq r$;

\item[\rm (b)] for $\mathcal{L}(^Nd_*(w,1);\ell_{r})$,   if $N \in \mathbb{N}$.
\end{enumerate}
\end{proposition}

\begin{proof}
(a) Fix $N \geq r$. Since $w \in \ell_r$, we may consider $\phi \in \mathcal{L}(^Nd_*(w,1))$ defined by $\phi(x_1,\ldots,x_N) = \sum_{i=1}^\infty x_1(i)\cdots x_N(i)$. Suppose that the Lindenstrauss-Bollob\'as theorem holds.

Take $0<\varepsilon<1$, $\eta(\varepsilon)$ and $\beta(\varepsilon)$ as in the definition  and let $\tilde x_1,\ldots,\tilde x_N \in B_{d_*(w,1)}$ be such that $\vert \phi(\tilde x_1,\ldots,\tilde x_N)\vert > \Vert \phi\Vert - \eta(\varepsilon)$. Then, there exists a multilinear mapping $\psi \in \mathcal{L}(^Nd_*(w,1))$, with $\Vert \phi - \psi\Vert<\varepsilon$, whose  Arens extensions attain the norm at  $(a_1,\ldots,a_N) \in B_{d^*(w,1)} \times \cdots \times B_{d^*(w,1)}$ and
\begin{equation*}
\Vert a_j - \tilde x_j\Vert_W < \beta(\varepsilon), \quad  \text{for all }\ 1\leq j \leq N.
\end{equation*}
With $\varepsilon$ sufficiently small, Lemma~\ref{lema 2} implies that  each $a_j$ satisfies  \eqref{propiedad extremal}. By Lemma~\ref{lema extremales}, $\psi(e_n,\ldots,e_n)=0$ for $n$ large enough. Since $\phi(e_n,\ldots,e_n) =1$ and $\Vert \phi - \psi\Vert < \varepsilon$, we get a contradiction and the statement follows.

(b) Fix $N \geq 1$. Since $w\in\ell_r$, the multilinear mapping $\Phi \in \mathcal{L}(^Nd_*(w,1); \ell_r)$ given by $\Phi(x_1,\dots,x_N) = (x_1(i)\cdots x_N(i))_{i \in \mathbb{N}}$, is well defined. Now, the result follows reasoning as before.
\end{proof}

Note that if $w\in \ell_2$, then $d_*(w,1)$ provides us with an example of a Banach space on which
the Lindenstrauss-Bollob\'as theorem fails for scalar multilinear forms of any degree other that 1.
Also, part (b) of the previous proof shows that, for $w\in\ell_r$,  the canonical inclusion $d_*(w,1) \hookrightarrow \ell_r$ cannot be approximated by linear mappings whose bitransposes are norm attaining. This example should be compared with
\cite[Example~6.3]{AAGM}.

Finally, we observe that in the proof of Proposition~\ref{no vale LB multilineal} it is enough to assume that just one of the  Arens extensions of $\psi$ attains its norm.

\bigskip

Now, we focus our attention on the polynomial version of the Lindenstrauss-Bollob\'as theorem. The following result extends Lemma~\ref{ppio mod max real} and its proof can be extracted  from that of \cite[Theorem~3.2]{JP}.

\begin{lemma} \label{generalizacion lema}
For the  real case, let $w$ be an admissible sequence in $\ell_N$, $N\ge 2$ and take $M$ the smallest natural number such that $w \in \ell_M$. Suppose that $p \in \mathcal{P}(^N d^*(w,1))$ attains its norm at
$a \in B_{d^*(w,1)}$ which satisfies condition \eqref{propiedad extremal} and let $\phi$ be the symmetric $N$-linear form associated to $p$.
\begin{enumerate}
\item[(i)] If $p(a) >0$\quad then\quad  $\limsup_n \phi(a,\dots,a,e_n,\overset{M}{\dots},e_n) \leq 0$.
\item[(ii)] If $p(a) <0$\quad  then\quad  $\liminf_n \phi(a,\dots,a,e_n,\overset{M}{\dots},e_n) \geq 0$.
\end{enumerate}
\end{lemma}

\begin{proposition} \label{no vale LB polinomial} Let $w$ be an admissible sequence and suppose $M$ is the smallest natural number such that $w\in \ell_M$ (we assume such an $M$ exists). There is no Lindenstrauss-Bollob\'as theorem in the following cases:

\begin{enumerate}
 \item[\rm (a)] for $\mathcal{P}(^Nd_*(w,1))$, for all $N \geq M$;

\item[\rm (b)]  for $\mathcal{P}(^Nd_*(w,1);\ell_{M})$, for all $N \in \mathbb{N}$.
\end{enumerate}
\end{proposition}

\begin{proof}
(a) \textit{The complex case.} Since $w \in \ell_M$, for any $N \geq M$ we can define $q \in \mathcal{P}(^Nd_*(w,1))$  by $q(x) = \sum_{i=1}^\infty x(i)^N$. The result is obtained by contradiction proceeding as in Proposition~\ref{no vale LB multilineal}.

\textit{The real case.}
Suppose that the Lindenstrauss-Bollob\'as theorem holds and  define $q \in \mathcal{P}(^Nd_*(w,1))$ by
$$
q(x) = x(1)^{N-M}\sum_{i=1}^\infty (-1)^i x(i)^M.
$$
Given $0<\varepsilon<1$, $\eta(\varepsilon)$ and $\beta(\varepsilon)$ as in the definition, take $\tilde x  \in B_{d_*(w,1)}$ such that $\vert q(\tilde x )\vert > \Vert q\Vert - \eta(\varepsilon)$. Then there exist $p \in \mathcal{P}(^Nd_*(w,1))$ and $a \in B_{d^*(w,1)}$ such that,
$$
\vert \overline{p}(a)\vert = \Vert \overline{p}\Vert= \Vert {p}\Vert, \quad \Vert a - \tilde x \Vert_W < \beta(\varepsilon) \quad \text{and} \quad \Vert p-q\Vert < \varepsilon^2.
$$
Let $\phi$ and $\psi$ be the symmetric $N$-linear forms associated to $\overline{p}$ and $\overline {q}$, respectively. By Lemma~\ref{lema 2}, we may choose  $\varepsilon$ sufficiently small so that $a$ satisfies \eqref{propiedad extremal} and $\|\phi - \psi\| < \varepsilon$.

If $\overline{p}(a) >0$,
Lemma~\ref{generalizacion lema} (i) implies that
$$
\limsup_n \phi(a,\ldots,a,e_n,\overset M \ldots, e_n) \leq 0.
$$
On the other hand, it is easy to see that
$\binom{N}{M} \psi(a,\ldots,a,e_n,\overset M \ldots, e_n) = (-1)^n a(1)^{N-M}$ and then
$$
\limsup_n \textstyle{\binom{N}{M}} \psi(a,\ldots,a,e_n,\overset M \ldots, e_n) = \vert a(1)\vert^{N-M}.
$$
Now, we proceed as in Proposition~\ref{no vale BP} and obtain $\|q\|=0$, which is a contradiction.

If $\overline{p}(a) <0$ the result follows using Lemma~\ref{generalizacion lema} (ii).

(b) \textit{The complex case}.
Note that $Q(x) = (x(i)^N)_{i \in \mathbb{N}}$  defines an element in
$\mathcal{P}(^Nd_*(w,1);\ell_M)$. Now, we proceed as in the Proposition~\ref{no vale LB multilineal}.

\textit{The real case}. Following the lines of the real case in (a), we can proceed as in the proof of Proposition~\ref{no vale BP} combining Lemma~\ref{lema 2} and Lemma~\ref{generalizacion lema}.
\end{proof}

As we did in the previous section, we may consider a Banach space $Z$ obtained by a renorming of $c_0$ so that $Z''$ is strictly convex. Then,
for any admissible sequence $w$ and any $N\in\N$,  the proof given in Proposition~\ref{no vale LB multilineal}~(b), remains true  for $\mathcal{L}(^Nd_*(w,1);Z)$.
The same happens with the proof of the complex case of Proposition~\ref{no vale LB polinomial}~(b) for  $\mathcal{P}(^Nd_*(w,1);Z)$.

\subsection*{Acknowledgements} We wish to thank Sunday Garc\'ia and Manolo Maestre for helpful conversations and comments. The third author is grateful to the Departamento  de An\'alisis Matem\'atico, Universitat de Valencia and its members for their hospitality during his visit in February 2012.

\end{document}